\documentclass[12pt]{amsart}
\usepackage[margin=1.2in]{geometry}
\usepackage{graphicx,latexsym}
\usepackage{comment}
\usepackage{amsfonts, amssymb, amsmath, amsthm, bm, hyperref}
\usepackage{tikz}
\usetikzlibrary{matrix,arrows}

\newcommand{\ds}{\displaystyle}
\newcommand{\N}{\mathbb{N}}
\newcommand{\Z}{\mathbb{Z}}
\newcommand{\R}{\mathbb{R}}

\newcommand{\EE}{\mathcal E}

\newcommand{\beq}{\begin{eqnarray}}
\newcommand{\eeq}{\end{eqnarray}}

\newcommand{\beqs}{\begin{eqnarray*}}
\newcommand{\eeqs}{\end{eqnarray*}}

\newtheorem{theorem}{Theorem}
\newtheorem{proposition}{Proposition}

\newtheorem{lemma}{Lemma}
\newtheorem{corollary}{Corollary}
\theoremstyle{definition}

\theoremstyle{remark}
\newtheorem{remark}{Remark}

\begin{document}
\title[Projective description of $\mathcal{E}^{\{\mathfrak{M}\}}(\Omega)$]{On the projective description of spaces of ultradifferentiable functions of Roumieu type}

\author[A. Debrouwere]{Andreas Debrouwere}
\thanks{A. Debrouwere was supported by  FWO-Vlaanderen through the postdoctoral grant 12T0519N}
\address{A. Debrouwere, Department of Mathematics: Analysis, Logic and Discrete Mathematics\\ Ghent University\\ Krijgslaan 281\\ 9000 Gent\\ Belgium}
\email{andreas.debrouwere@UGent.be}

\author[B. Prangoski]{Bojan Prangoski}
\thanks{B. Prangoski was partially supported by the Macedonian Academy of Sciences and Arts through the project ``Microlocal analysis and Applications''}
\address{B. Prangoski, Faculty of Mechanical Engineering\\ University Ss. Cyril and Methodius \\ Karpos II bb \\ 1000 Skopje \\ Macedonia}
\email{bprangoski@yahoo.com}

\author[J. Vindas]{Jasson Vindas}
\thanks {J. Vindas was supported by Ghent University through the BOF-grants 01J11615 and 01J04017.}
\address{J. Vindas, Department of Mathematics: Analysis, Logic and Discrete Mathematics\\ Ghent University\\ Krijgslaan 281\\ 9000 Gent\\ Belgium}
\email{jasson.vindas@UGent.be}

\begin{abstract}
We provide a projective description of the space $\mathcal{E}^{\{\mathfrak{M}\}}(\Omega)$  of ultradifferentiable functions of Roumieu type, where $\Omega$ is an arbitrary open set in $\R^d$ and $\mathfrak{M}$ is a weight matrix satisfying the analogue of Komatsu's condition $(M.2)'$. In particular, we obtain in a unified way projective descriptions of ultradifferentiable classes defined via a single weight sequence (Denjoy-Carleman approach) and via a weight function (Braun-Meise-Taylor approach) under considerably weaker assumptions  than in earlier versions of these results.

\end{abstract}
\maketitle
\section{Introduction}
In his seminal work \cite{Komatsu3}, Komatsu gave an explicit system of seminorms generating the topology of the space $\EE^{\{M\}}(\Omega)$ of ultradifferentiable functions of Roumieu type (for short, a projective description of $\EE^{\{M\}}(\Omega)$), where $\Omega$ is an arbitrary open set of $\R^d$ and  $M$ is a non-quasianalytic weight sequence satisfying the conditions $(M.1)$ and $(M.2)'$ \cite[Proposition 3.5]{Komatsu3}. In \cite[Proposition 4.8]{D-V}, the first and the third authors relaxed the non-quasianalyticity assumption on $M$ to $\sup_{p \in \N} p M_p^{-1/p} < \infty$. Similarly, a projective description of the space $\EE^{\{\omega\}}(\Omega)$ was implicitly given in \cite[Section 3]{H-M}, where $\omega$ is a weight function in the sense of Braun, Meise and Taylor \cite{B-M-T} that satisfies $\omega(t) = O(t)$. Projective descriptions are indispensable in the study of spaces of vector-valued ultradifferentiable functions of Roumieu type \cite{Komatsu3, D-V}, e.g., for achieving completed tensor product representations of such spaces. 

The goal of this article is to provide a projective description of spaces of ultradifferentiable functions of Roumieu type defined via a weight matrix \cite{R-S}. This approach leads to a unified treatment of ultradifferentiable classes defined via a single weight sequence  and via a weight function, but also comprises other spaces, e.g., the union of all Gevrey spaces. For the two standard classes we obtain projective descriptions under  much weaker assumptions  than in the above mentioned works; see Corollary \ref{cor-1} and Corollary \ref{cor-2}.

\section{Spaces of ultradifferentiable functions of Roumieu type}
Let $M= (M_p)_{p \in \N}$ be a sequence of positive numbers (a \emph{weight sequence}). We consider the following three conditions on $M$:
\begin{itemize}
\item[$(M.0)\:$]  $M_p \geq ch^p$, $p \in \N$, for some $c,h > 0$;
\item[$(M.1)\:$]  $M_{p}^{2} \leq M_{p-1} M_{p+1}$, $p \in\Z_+$;
\item[$(M.2)'$]  $\ds M_{p+1} \leq CH^{p} M_p$, $p\in \N$, for some $C,H > 0$.
\end{itemize}
The conditions $(M.1)$ and $(M.2)'$ are denoted by $(M_{\operatorname{lc}})$ and $(M_{\operatorname{dc}})$, respectively, in \cite{R-S}. We use here the standard notation from \cite{Komatsu}.   The relation $M \subset N$ between two weight sequences means that there are $C,h>0$ such that $M_p\leq Ch^{p}N_{p}$ for all $p\in\N$.  The stronger relation $M \prec N$ means that the latter inequality remains valid for every $h>0$ and a suitable $C=C_{h}>0$. We set $M_\alpha = M_{|\alpha|}$, $\alpha \in \N^d$.

For $h > 0$ and a regular compact set $K \Subset \R^d$ (i.e., $K=\overline{\mathrm{int}\, K}$) we write $\mathcal{E}^{M,h}(K)$ for the Banach space consisting of all $\varphi \in C^\infty(K)$\footnote{We define $C^\infty(K)$ as the space consisting of all $\varphi \in C^\infty(\operatorname{int} K)$ such that $\partial^{\alpha}\varphi$ extends to a continuous function on $K$ for each $\alpha \in \N^d$.} such that
$$
\|\varphi\|_{\mathcal{E}^{M,h}(K)}:=\sup_{\alpha\in\N^d}\frac{h^{|\alpha|}\|\partial^{\alpha}\varphi\|_{L^{\infty}(K)}}{M_{\alpha}}<\infty.
$$
We set
$$
\EE^{\{M\}}(K):= \varinjlim_{h \to 0^+} \EE^{M,h}(K).
$$
Given an open set $\Omega \subseteq \R^d$, we define the space of ultradifferentiable functions of Roumieu type (of class $\{M\}$) on $\Omega$ as 
\beqs
\EE^{\{M\}}(\Omega):=\varprojlim_{K\Subset \Omega} \EE^{\{M\}}(K).
\eeqs

Next, we introduce weight matrices and the associated spaces of ultradifferentiable functions \cite{R-S}. A \emph{weight matrix} is a sequence $\mathfrak{M} = (M^n)_{n \in \N}$ consisting of weight sequences $M^n$ such that $M^n \leq M^{n+1}$ for all $n \in \N$. We consider the following condition on $\mathfrak{M}$: 
 \begin{itemize}
\item[$\{\mathfrak{M}.2\}'$] $\forall n \in \N ~ \exists m \in \N ~ \exists C,H >0 ~ \forall p  \in \N \, : \, M_{p+1}^{n} \leq CH^{p} M_{p}^m$.
	\end{itemize}
The condition $\{\mathfrak{M}.2\}'$ is denoted by  $(\mathfrak{M}_{\{\operatorname{dc}\}})$ in \cite{R-S}. Given a regular compact set $K \Subset \R^d$, we denote
$$
\EE^{\{\mathfrak{M}\}}(K):= \varinjlim_{n \in \N} \EE^{\{M_n\}}(K).
$$
Given an open set $\Omega \subseteq \R^d$, we define the space of ultradifferentiable functions of Roumieu type (of class $\{\mathfrak{M}\}$) on $\Omega$ as 
\beqs
\EE^{\{\mathfrak{M}\}}(\Omega):=\varprojlim_{K\Subset \Omega} \EE^{\{\mathfrak{M}\}}(K).
\eeqs

Finally, we introduce spaces of ultradifferentiable functions defined via a weight function in the sense of Braun, Meise and Taylor \cite{B-M-T} and explain how they fit into the weight matrix approach; see \cite[Section 5]{R-S} for more details. By a \emph{weight function} we mean a continuous increasing function $\omega: [0,\infty) \rightarrow [0,\infty)$ with $\omega_{|[0,1]} \equiv 0$ and satisfying the following three properties: 
	\begin{itemize}
		\item[$(\alpha)$] $\omega(2t) = O(\omega(t))$;
		\item[$(\gamma_0)$] $\log t = o(\omega(t))$;
		\item[$(\delta) $] $\phi = \phi_\omega: [0, \infty) \rightarrow [0, \infty)$, given by $\phi(x) = \omega(e^{x})$, is convex.
	\end{itemize}
Note that $\phi^{*}$ is increasing and convex, $\phi^*(0) = 0$, $(\phi^*)^* = \phi$, $\phi^*(y)/y$ is increasing on $[0,\infty)$ and $\phi^*(y)/y \rightarrow \infty$ as $y \to \infty$. For $\rho > 0$ and a regular compact set $K \Subset \R^d$ we write $\mathcal{E}^{\omega,\rho}(K)$ for the Banach space consisting of all $\varphi \in C^\infty(K)$ such that
$$
\|\varphi\|_{\mathcal{E}^{\omega,\rho}(K)}:=\sup_{\alpha\in\N^d} \|\partial^{\alpha}\varphi\|_{L^{\infty}(K)} \exp\left(-\frac{1}{\rho}\phi^*(\rho|\alpha|) \right)<\infty.
$$

Given an open set $\Omega \subseteq \R^d$, we define the space of ultradifferentiable functions of Roumieu type (of class $\{\omega\}$) on $\Omega$ as 
\beqs
\EE^{\{\omega\}}(\Omega):=\varprojlim_{K\Subset \Omega} \varinjlim_{\rho \to \infty} \EE^{\omega,\rho}(K).
\eeqs
We associate to $\omega$ to the weight matrix $\mathfrak{M}_{\omega} = (M_{\omega}^{n})_{n \in \N}$, where the weight sequence $M_\omega^{n}$ is defined as
$$
M^n_{\omega,p} := \exp\left(\frac{1}{n}\phi^*(np) \right), \qquad p \in \N.
$$
Note that each $M_\omega^{n}$ satisfies $(M.0)$ and $(M.1)$. Furthermore, $\mathfrak{M}_{\omega}$ satisfies $\{\mathfrak{M}.2\}'$  and  $\EE^{\{\omega\}}(\Omega) = \EE^{\{\mathfrak{M}_{\omega}\}}(\Omega)$ as locally convex spaces \cite[Corollary 5.15]{R-S}.

\section{Projective description of $\mathcal{E}^{\{\mathfrak{M}\}}(\Omega)$}\label{apptopol}
Given a weight matrix $\mathfrak{M}$, we define $V(\mathfrak{M})$ as the set of all those weight sequences $N$ such that  $M^n \prec N$ for all $n \in \N$. The next theorem is the main result of this article.
\begin{theorem}\label{main}
Let $\Omega \subseteq \R^d$ be open and let $\mathfrak{M}$ be a weight matrix satisfying $\{\mathfrak{M}.2\}'$. A function $\varphi \in C^\infty(\Omega)$ belongs to $\EE^{\{\mathfrak{M}\}}(\Omega)$ if and only if
$$
\| \varphi\|_{\mathcal{E}^{N,1}(K)} = \sup_{\alpha\in\N^d}\frac{\|\partial^{\alpha}\varphi\|_{L^{\infty}(K)}}{N_{\alpha}}<\infty
$$
for all $K \Subset \Omega$ and $N \in V(\mathfrak{M})$. Moreover,  the locally convex topology of $\EE^{\{\mathfrak{M}\}}(\Omega)$ is generated by the system of seminorms $\{ \| \, \cdot \, \|_{\mathcal{E}^{N,1}(K)} \, | \, K \Subset \Omega, N \in V(\mathfrak{M})\}$.
\end{theorem}
\begin{remark}\label{main-remark}
Let  $\mathfrak{M}$ be a weight matrix satisfying $\{\mathfrak{M}.2\}'$ and suppose that each weight sequence $M^n$ satisfies $(M.0)$ and $(M.1)$. Obviously, every element of $V(\mathfrak{M})$ automatically satisfies $(M.0)$. Define $V^*(\mathfrak{M})$  as the set of all $N\in V(\mathfrak{M})$ for which $(M.1)$ holds. Then, Theorem \ref{main} still holds true if we replace  $V(\mathfrak{M})$ by $V^*(\mathfrak{M})$. This follows from the fact that for each $N \in V(\mathfrak{M})$ its log-convex minorant $N^c$ belongs to $V^*(\mathfrak{M})$ and satisfies $N^c \leq N$.
\end{remark}
Before we prove Theorem \ref{main}, let us show how it entails the projective description of the spaces $\mathcal{E}^{\{M\}}(\Omega)$ and $\mathcal{E}^{\{\omega\}}(\Omega)$. Following Komatsu \cite{Komatsu3}, we denote by $\mathfrak{R}$ the family of all non-decreasing sequences $r = (r_j)_{j \in \N}$ of positive numbers such that $r_j \to \infty$ as $j \to \infty$. The next result generalizes \cite[Proposition 3.5]{Komatsu3} and \cite[Proposition 4.8]{D-V}. 

\begin{corollary}\label{cor-1}
Let $\Omega \subseteq \R^d$ be open and let $M$ be a weight sequence satisfying $(M.2)'$.  A function $\varphi \in C^\infty(\Omega)$ belongs to $\EE^{\{M\}}(\Omega)$ if and only if
$$
\| \varphi\|_{\mathcal{E}^{M,r}(K)} := \sup_{\alpha\in\N^d}\frac{\|\partial^{\alpha}\varphi\|_{L^{\infty}(K)}}{M_{\alpha} \prod_{j = 0}^{|\alpha|} r_j}<\infty
$$
for all $K \Subset \Omega$ and $r \in \mathfrak{R}$. Moreover,  the locally convex topology of $\EE^{\{M\}}(\Omega)$ is generated by the system of seminorms $\{ \| \, \cdot \, \|_{\mathcal{E}^{M,r}(K)} \, | \, K \Subset \Omega, r \in \mathfrak{R}\}$.
\end{corollary}
\begin{proof}
This follows from Theorem \ref{main} (applied to the constant weight matrix $\mathfrak{M} = (M)_{n\in\N}$) and \cite[Lemma 3.4]{Komatsu3}.
\end{proof}
Given a weight function $\omega$, we define $V(\omega)$ as the set consisting of all weight functions $\sigma$ such that $\sigma = o(\omega)$.

\begin{corollary}\label{cor-2}
Let $\Omega \subseteq \R^d$ be open and let $\omega$ be a weight function. A function $\varphi \in C^\infty(\Omega)$ belongs to $\EE^{\{\omega\}}(\Omega)$ if and only if
$$
\| \varphi\|_{\mathcal{E}^{\sigma,1}(K)} = \sup_{\alpha\in\N^d} \|\partial^{\alpha}\varphi\|_{L^{\infty}(K)} e^{-\phi_\sigma^*(|\alpha|)}<\infty 
$$
for all $K \Subset \Omega$ and $\sigma \in V(\omega)$. Moreover,  the locally convex topology of $\EE^{\{\omega\}}(\Omega)$ is generated by the system of seminorms $\{ \| \, \cdot \, \|_{\mathcal{E}^{\sigma,1}(K)} \, | \, K \Subset \Omega, \sigma \in V(\omega)\}$.
\end{corollary}
\begin{proof}
By Theorem \ref{main} and Remark \ref{main-remark} (applied to the weight matrix $\mathfrak{M}_\omega$) it suffices to show that
\begin{itemize}
\item[$(i)$] $\forall \sigma \in V(\omega) \, : \,  M^1_\sigma \in V(\mathfrak{M}_\omega)$.
\item[$(ii)$] $\forall N \in V^*(\mathfrak{M}_\omega) \, \exists \sigma \in V(\omega) \, : \, M^1_\sigma \subset N$.
\end{itemize}
The first statement is obvious. We now show the second one. Let $N \in V^*(\mathfrak{M}_\omega)$ be arbitrary. Consider the associated function of $N$
$$
\omega_N(t) = \sup_{p \in \N} \log \frac{t^pN_0}{N_p}, \qquad t \geq 0.
$$
Then, $\omega_N = o(\omega)$. By \cite[Lemma 1.7 and Remark 1.8]{B-M-T}, there is a weight function $\sigma \in V(\omega)$ such that $\omega_N = o(\sigma)$. Since $\omega_{M^1_\sigma} \asymp \sigma$ \cite[Lemma 5.7]{R-S}, we obtain that
$$
\omega_N(t) \leq \omega_{M^1_\sigma}(t) + C, \qquad t \geq 0.
$$
Since both $N$ and $M^1_\sigma$ satisfy $(M.0)$ and $(M.1)$, the latter inequality yields that $M^1_\sigma \subset N$ \cite[Lemma 3.8]{Komatsu}.
\end{proof}
We now turn to the proof of Theorem \ref{main}.
We use the same idea as in  Komatsu's proof of  \cite[Proposition 3.5]{Komatsu3}.  Fix a weight matrix $\mathfrak{M}$ satisfying $\{\mathfrak{M}.2\}'$. Since any open set $\Omega \subseteq \R^d$ admits an exhaustion by compact sets that are finite unions of regular connected compact sets $K$ with smooth boundary\footnote{This follows from the existence of a positive smooth exhausting function on $\Omega$ \cite[Proposition 2.28]{lee} and Sard's theorem \cite[Theorem 6.10]{lee}.} (in particular,  $\operatorname{int} K$ is a Lipschitz domain), Theorem \ref{main} follows from the next result.
\begin{theorem}\label{main-2}
Let $K \Subset \R^d$ be a regular compact set such that $\operatorname{int} K$ is a Lipschitz domain and let $\mathfrak{M}$ be a weight matrix satisfying $\{\mathfrak{M}.2\}'$. A function $\varphi \in C^\infty(K)$ belongs to $\EE^{\{\mathfrak{M}\}}(K)$ if and only if
$\| \varphi\|_{\mathcal{E}^{N,1}(K)}<\infty$
for all $N \in V(\mathfrak{M})$. Moreover,  the locally convex topology of $\EE^{\{\mathfrak{M}\}}(K)$ is generated by the system of seminorms $\{ \| \, \cdot \, \|_{\mathcal{E}^{N,1}(K)} \, | \,  N \in V(\mathfrak{M})\}$.
\end{theorem}
The rest of this article is devoted to the proof of Theorem \ref{main-2}. We start with the following technical lemma (cf.\ \cite[Lemma 3.4]{Komatsu}).
\begin{lemma}\label{algebraic} Let $(a_p)_{p \in \N}$ be a sequence of positive numbers.
\begin{itemize}
\item[$(i)$] $\displaystyle \sup_{p \in \N} \frac{h^p a_p}{M^n_p} < \infty$ for some $h > 0$ and $n \in \N$ if and only if $\displaystyle \sup_{p \in \N} \frac{a_p}{N_p} < \infty$ for all $N \in V(\mathfrak{M})$.
\item[$(ii)$] $\displaystyle \sup_{p \in \N} a_pN_p < \infty$ for some $N \in V(\mathfrak{M})$ if and only if $\displaystyle \sup_{p \in \N} \frac{a_pM^n_p}{h^p} < \infty$ for all $h > 0$ and $n \in \N$.
\end{itemize}
\end{lemma}
\begin{proof} 
The direct implications are clear. We now show the converse ones.

$(i)$ Suppose that $\sup_{p \in \N} a_p/(n^pM^n_p) = \infty$ for all $n \in \N$. Choose a strictly increasing sequence $(p_n)_{n \in \N}$ of natural numbers with $p_0 = 0$ and 
$$
\frac{a_{p_n}}{n^{p_n}M^n_{p_n}} \geq n, \qquad n \in \Z_+.
$$
Define $N_p = n^{p}M^n_{p}$ if $p_n \leq p < p_{n+1}$. Then, $N = (N_p)_{p \in \N}$ belongs to  $V(\mathfrak{M})$ but $\sup_{p \in \N} a_p/N_p = \infty$, a contradiction.

$(ii)$ For each $n \in \N$ there is $C_n > 0$ such that 
$$
n^pM^n_p \leq \frac{C_n}{a_p}, \qquad p \in \N.
$$
Define
$$
N_p = \sup_{n \in \N} \frac{n^p M^n_p}{C_n}, \qquad p \in \N.
$$
Then, $N = (N_p)_{p \in \N}$ belongs to  $V(\mathfrak{M})$ and $\sup_{p \in \N} a_pN_p < \infty$.
\end{proof}

 A set $A \subseteq \R^d$ is said to be $\emph{quasiconvex}$ if there exists $C>0$ such that any two points $x,y\in A$ can be joined by a  curve in $A$ with length at most $C|x-y|$. This notion was introduced by Whitney \cite{whitney2} under the name property $(P)$. The closure of a quasiconvex open set is again quasiconvex \cite[Lemma 2]{whitney2}.  Moreover, every  bounded Lipschitz domain is quasiconvex \cite[Section 2.5]{B-B}.

Let $K \Subset \R^d$ be a regular compact set such that $\operatorname{int} K$ is quasiconvex. For $n\in\N$ we denote by $C^n(K)$ the vector space of all $\varphi\in C^n(\operatorname{int}K)$ such that $\partial^{\alpha}\varphi$ extends to a continuous function on $K$ for each $|\alpha|\leq n$;  it is a Banach space when endowed with the norm $\sup_{|\alpha|\leq n}\|\partial^{\alpha}\varphi\|_{L^{\infty}(K)}$. By \cite[Theorem, p.\ 485]{whitney2}, the space $C^n(K)$ is topologically isomorphic to the Banach space of Whitney jets of order $n$ on $K$ \cite{whitney1}. Let $R > 0$ be such that $K \Subset B(0,R)$. Whitney's extension theorem \cite[Theorem I]{whitney3} yields the existence of a  continuous linear extension operator $C^n(K) \rightarrow C^n(\overline{B}(0,R))$, that is, a continuous linear right inverse of the restriction mapping $C^n(\overline{B}(0,R)) \rightarrow C^n(K)$. The latter implies that the inclusion mapping $C^{n+1}(K) \rightarrow C^{n}(K)$ is compact. A standard argument (cf.\ \cite[Proposition 2.2]{Komatsu})
 therefore gives the following result.

\begin{lemma}\label{lemma-1} Let $K \Subset \R^d$ be a regular compact set such that $\operatorname{int} K$ is quasiconvex. Then, $\EE^{\{\mathfrak{M}\}}(K)$ is a $(DFS)$-space.
\end{lemma} 

Next, we show a structural result for the dual of $\EE^{\{\mathfrak{M}\}}(K)$; this is the crux of the proof of Theorem \ref{main-2}. We need some preparation. Given a Banach space $E$, a weight sequence $M = (M_p)_{p \in \N}$ and $h>0$, we define $\Lambda^{M,h}(E)$ as the Banach space consisting of all multi-indexed sequences $e = (e_{\alpha})_{\alpha \in \N^d} \in E^{\N^d}$ such that
$$
\|e\|_{\Lambda^{M,h}(E)} := \sup_{\alpha \in \N^d} \frac{h^{|\alpha|} \|e_\alpha\|_E}{M_\alpha} < \infty.
$$
We define the $(LB)$-space
$$
\Lambda^{\{\mathfrak{M}\}}(E) := \varinjlim_{n \in \N} \varinjlim_{h \to 0^+} \Lambda^{M_n,h}(E)
$$
and the Fr\'echet space
$$
\Lambda^{' \{\mathfrak{M}\}}(E) := \varprojlim_{n \in \N} \varprojlim_{h \to 0^+} \Lambda^{1/M_n,1/h}(E).
$$
The dual of $\Lambda^{' \{\mathfrak{M}\}}(E)$ may be identified with $\Lambda^{\{\mathfrak{M}\}}(E')$; the dual pairing under this identification is given by 
$$
\langle e',e \rangle = \sum_{\alpha \in \N^d} \langle e'_\alpha,e_\alpha \rangle, \qquad e' \in \Lambda^{\{\mathfrak{M}\}}(E'), e \in \Lambda^{'\{\mathfrak{M}\}}(E).
$$

\begin{proposition}\label{vstkln135} Let $K \Subset \R^d$ be a regular compact set such that $\operatorname{int} K$ is a Lipschitz domain and let $\mathfrak{M}$ be a weight matrix satisfying $\{\mathfrak{M}.2\}'$.
	Let $B$ be an equicontinuous subset of $(\EE^{ \{\mathfrak{M}\}}(K))'$. There exist $N \in V(\mathfrak{M})$ and $C > 0$ such that for each $T \in B$ there is a family $\{F_{\alpha,T}\in L^2(K)|\,\alpha\in\N^d\}$ satisfying
	\beq\label{kvstld135}
\sup_{\alpha\in\N^d}\|F_{\alpha,T}\|_{L^2(K)}N_{\alpha} \leq C
	\eeq
	and
	\beq\label{stvknr157}
	\langle T,\varphi\rangle=\sum_{\alpha\in\N^d}\int_K F_{\alpha,T}(x)\partial^{\alpha}\varphi(x)dx, \qquad \varphi\in\EE^{\{\mathfrak{M}\}}(K).
	\eeq
\end{proposition}

\begin{proof} 
We claim that the continuous linear mapping
$$
S: \Lambda^{' \{\mathfrak{M}\}}(L^2(K)) \rightarrow (\EE^{ \{\mathfrak{M}\}}(K))'_\beta, \quad (F_\alpha)_{\alpha \in \N^d} \mapsto \sum_{\alpha \in \N^d} (-1)^{|\alpha|} \partial^{\alpha}F_\alpha
$$
is surjective. Before we prove the claim, let us show how it implies the result. By Lemma \ref{algebraic}$(ii)$, it suffices to show that for every bounded subset $B$ of $(\EE^{ \{\mathfrak{M}\}}(K))'_\beta$ (in particular, for every equicontinuous subset $B$ of $(\EE^{ \{\mathfrak{M}\}}(K))'$) there is a bounded subset $A$ of $\Lambda^{' \{\mathfrak{M}\}}(L^2(K))$ such that $S(A) = B$. Since $(\EE^{ \{\mathfrak{M}\}}(K))'_\beta$ is a Fr\'echet-Montel space (Lemma \ref{lemma-1}), this follows from the following general fact \cite[Corollary 26.22]{M-V}: Let $T: E \rightarrow F$ be a surjective continuous linear mapping between a Fr\'echet space $E$ and a Fr\'echet-Montel space $F$. Then, for every bounded subset $B$ of  $F$ there is a bounded subset $A$ of $E$ such that $S(A) = B$. We now prove the claim. To this end, it suffices to show that the transpose of $S$ is injective and has weak-$\ast$ closed range. By the remarks preceding this proposition and the fact that $\EE^{ \{\mathfrak{M}\}}(K)$ is reflexive (Lemma \ref{lemma-1}), we may identify the transpose of $S$ with the mapping
$$
S^t: \EE^{ \{\mathfrak{M}\}}(K) \rightarrow \Lambda^{\{\mathfrak{M}\}}(L^2(K)), \, \varphi \rightarrow (\partial^\alpha\varphi)_{\alpha \in \N^d}. 
$$
This mapping is clearly injective. We now show that it has weak-$\ast$ closed range. Let $(\varphi_j)_{j}$ be a net in $\EE^{ \{\mathfrak{M}\}}(K)$ and $F = (F_\alpha)_{\alpha \in \N^d} \in \Lambda^{\{\mathfrak{M}\}}(L^2(K))$ such that $S^t(\varphi_j) \rightarrow F$ in the weak-$\ast$ topology on  $\Lambda^{\{\mathfrak{M}\}}(L^2(K))$. In particular, $\partial^\alpha \varphi_j \rightarrow  F_\alpha$ in $L^1_{\operatorname{loc}}(\operatorname{int} K)$ for all $\alpha \in \N^d$. Consequently, $\partial^\alpha F_0 = F_\alpha \in L^2(\operatorname{int} K)$ (the derivatives should be interpreted in the sense of  distributions). By the Sobolev embedding theorem \cite[Theorem 4.12 Part II]{adams}, there is $k \in \N$ such that the continuous embedding $H^k(\operatorname{int} K) \rightarrow C(K)$ hodls, where $H^k(\operatorname{int} K)$ denotes the Sobolev space of order $k$. Since $\partial^\alpha F_0  \in H^k(\operatorname{int} K)$ for all $\alpha \in \N^d$, we obtain that $F_0 \in C^\infty(K)$ and 
$$
\| \partial^\alpha F_0 \|_{L^\infty(K)} \leq D \max_{|\beta| \leq k} \| \partial^{\alpha+\beta} F_0 \|_{L^2(\operatorname{int} K)} = D \max_{|\beta| \leq k} \|  F_{\alpha+\beta} \|_{L^2(\operatorname{int} K)}, \qquad   \alpha \in \N^d,
$$
for some $D > 0$. Pick $0 < h \leq 1$ and $n \in \N$ such that $F \in  \Lambda^{M_n, h}(L^2(K))$. Condition $\{ \mathfrak{M}.2\}'$ implies  that there are $m \in \N$ and $C,H > 0$ such that $\max_{0 \leq j \leq k}M^n_{p+j} \leq CH^pM^m_p$ for all $p \in \N$. Hence,
$$
\| F_0 \|_{\mathcal{E}^{M_m, h/H}(K)} \leq DCh^{-k} \| F\|_{ \Lambda^{M_n, h}(L^2(K))}.
$$
This shows that $F_0 \in \mathcal{E}^{\{\mathfrak{M}\}}(K)$ and thus $F = (F_{\alpha})_{\alpha \in \N^d} =  (\partial^{\alpha}F_0)_{\alpha \in \N^d}  \in \operatorname{Im} S^t$.
 \end{proof}
\begin{remark}
One can also use the dual Mittag-Leffler lemma \cite[Lemma 1.4]{Komatsu} in the same way as in the proof of \cite[Theorem 3.2$(ii)$]{D-P-P-V}  (see also the proof of \cite[Proposition 8.6]{Komatsu}) to show Proposition \ref{vstkln135}. 
\end{remark}
\begin{proof}[Proof of Theorem \ref{main-2}]
The first statement  is a  consequence of Lemma \ref{algebraic}$(i)$. Moreover, it is clear that for each $N \in V(\mathfrak{M})$ the seminorm $\| \, \cdot \, \|_{\mathcal{E}^{N,1}(K)}$ is continuous on  $\EE^{\{\mathfrak{M}\}}(K)$. We now show that for every seminorm $q$ on $\EE^{\{\mathfrak{M}\}}(K)$ there are $N \in V(\mathfrak{M})$ and $C > 0$ such that
$$
q(\varphi) \leq C \| \varphi\|_{\mathcal{E}^{N,1}(K)}, \qquad \varphi \in \mathcal{E}^{\{\mathfrak{M}\}}(K).
$$
Choose an equicontinuous subset $B$ of $\EE'^{\{\mathfrak{M}\}}(K)$ such that
$$
q(\varphi) = \sup_{T \in B} |\langle T, \varphi \rangle|, \qquad \varphi \in \mathcal{E}^{\{\mathfrak{M}\}}(K).
$$
By Proposition \ref{vstkln135}, there exist $N \in V(\mathfrak{M})$ and $C > 0$ such that for each $T \in B$ there is a family $\{F_{\alpha,T}\in L^2(K)|\,\alpha\in\N^d\}$ satisfying \eqref{kvstld135} and \eqref{stvknr157}. Set $L = (N_p/2^p)_{p \in \N} \in V(\mathfrak{M})$.  For all $\varphi \in \mathcal{E}^{\{\mathfrak{M}\}}(K)$ it holds that
\begin{align*}
&q(\varphi) \leq \sup_{T \in B} \sum_{\alpha\in\N^d} \int_K |F_{\alpha,T}(x)\|\partial^{\alpha}\varphi(x)|dx \leq \sup_{T \in B}  \sum_{\alpha\in\N^d}  \| F_{\alpha,T}\|_{L^2(K)} \| \partial^{\alpha} \varphi \|_{L^2(K)} \\
&\leq C|K|^{1/2} \sum_{\alpha \in \N^d} \frac{ \| \partial^{\alpha} \varphi \|_{L^\infty(K)}}{N_\alpha} \leq 2^dC|K|^{1/2}\| \varphi\|_{\mathcal{E}^{L,1}(K)}.
\qedhere
\end{align*}
\end{proof}

\end{document}